\documentclass[a4,12pt]{article}

\usepackage{amssymb,amsmath,amsfonts}
\usepackage{multirow,ulem,dsfont}
\usepackage{amsthm,mathtools,stmaryrd}

\newtheorem{thm}{Theorem}
\newtheorem{defi}{Definition}
\newtheorem{prop}{Proposition}
\newtheorem{lem}{Lemma}
\newtheorem{cor}{Corollary}

\newtheorem{rmrk}{Remark}

\newcommand{\IN}{\mathds{N}}

\newcommand{\IR}{\mathds{R}}
\newcommand{\IC}{\mathds{C}}

\newcommand{\I}[2]{\mathcal{I}^{#1,\rho}_{#2}}
\newcommand{\C}[2]{\mathcal{C}^{#1,\rho}_{#2}}
\newcommand{\D}[2]{\mathcal{D}^{#1,\rho}_{#2}}
\newcommand{\Hd}[2]{\mathcal{H}^{#1}_{#2}}

\newcommand{\dis}[1]{\displaystyle{#1}}
\newcommand{\up}[1]{\textsuperscript{#1}}

\renewcommand{\geq}{\geqslant}
\renewcommand{\leq}{\leqslant}

\begin{document}

\begin{center}
{\Large \textbf{Generalized Taylor formulas involving generalized fractional derivatives}}\\
\bigskip
Mondher Benjemaa
\end{center}

\begin{center}
\textit{Laboratory of Stability and Control of Systems and nonlinear PDEs, \\ Sfax University, Tunisia}
\end{center}

\hrule \vskip 5mm

\noindent \textbf{Abstract}\\

In this paper, we establish a generalized Taylor expansion of a given function $f$ in the form\\

$\dis{f(x) = \sum_{j=0}^m c_j^{\alpha,\rho}\left(x^\rho-a^\rho\right)^{j\alpha} + e_m(x)}$\\

\noindent with $m\in \IN$, $c_j^{\alpha,\rho}\in \IR$, $x>a$ and $0< \alpha \leq 1$. In case $\rho = \alpha = 1$, this expression coincides with the classical Taylor formula. The coefficients $c_j^{\alpha,\rho}$, $j=0,\dots,m$ as well as an estimation of $e_m(x)$ are given in terms of the generalized Caputo-type fractional derivatives. Some applications of these results for approximation of functions and for solving some fractional differential equations in series form are given in illustration.\\



\hrule

\section{Introduction}\label{intro}
Fractional calculus has intensively developed since its introduction in the seventies and is nowadays a vividly growing research field. The basic idea behind the fractional calculus is to extend to real or complex orders the classical integrals and derivatives involving integer orders. Consequently, it provides a useful and powerful tool to solve differential and integral equations as well as various mathematical and physical problems involving nonlocal effects or memory effects, such as quantum mechanics, biophysics, fluid mechanics, control theory and dynamical systems, diffusive transport equations, viscoelasticity, signal processing, probability and statistics, and so on. In the literature, several different fractional derivatives have been introduced, e.g. Riemann-Liouville, Caputo, Hadamard, Erd\'elyi-Kober, Hadamard, Gr\"unwald-Letnikov, Marchaud and Riesz among others \cite{Her14,Hil00,Mach11,Old74,Pod99,Sam93}.\\

Recently the author in \cite{Kat11} has introduced a new fractional integral which generalizes into a single form the Riemann-Liouville and the Hadamard integrals. Later on, he has shown that the generalized fractional integral operator is invertible and he has introduced in \cite{Kat14} a new fractional derivative, which generalizes the Riemann-Liouville and the Hadamard derivatives. More recently, the authors in \cite{Jar17} have studied the generalized fractional derivative in Caputo sens. Particularly, they have established that the generalized Caputo-type fractional derivative converges toward the Caputo-type Riemann-Liouville and the Caputo-type Hadamard derivatives when a parameter (denoted $\rho$) goes to zero and one respectively (see \cite[Theorem 3.11]{Jar17} or Theorem \ref{thm_cv_rho_caputo} hereafter). We propose in this work to investigate the generalized Taylor formulas involving these generalized Caputo-type fractional derivatives.\\

The paper is organized in the following way. After some definitions and notations in section \ref{sec_Notation}, we establish in section \ref{sec_taylor} the Taylor expansion of a given function by means of its generalized Caputo-type fractional derivatives. In section \ref{sec_muntz}, we provide an approximation of a given function in terms of M\"untz polynomials. Finally, we apply these results to find the solutions of some fractional differential equations (fde) in series form. We also provide in Appendix an iterative algorithm to compute the $j^{th}$ generalized fractional derivative function of the classical derivatives up to the $j^{th}$ order.


\section{Definitions and notations}\label{sec_Notation}
Throughout this paper, $\IN_0$ will denote the set of non-negative integers, $a$ and $b$ will denote two given real numbers such that $-\infty < a < b < +\infty$ and $\rho$ a positive real number. Unless otherwise mentioned, $\alpha$ will denote a complex number such that $Re(\alpha)>0$. We will use the notation $\lfloor x\rfloor$ to design the integer part of a real number $x$, that is the greatest integer less than or equal to $x$. We also define $\lfloor x \rceil = \lfloor x\rfloor +1$ if $x\not \in \IN_0$ and $\lfloor x \rceil = x$ if $x\in \IN_0$. The set of absolute continuous functions on $[a,b]$ will be denoted $AC[a,b]$. Then we define
\begin{align*}
AC^n_\gamma :=\Big\{f:[a,b]\rightarrow \IC, \ \gamma^{n-1} f \in AC[a,b]\Big\}
\end{align*}
with $\gamma := x^{1-\rho}\,\dfrac{d}{dx}$ and $AC^1_\gamma = AC[a,b]$.\\

The left-sided generalized fractional integral $\I{\alpha}{a^+}f$ of order $\alpha$ is defined for any real number $x>a$ by:
\begin{align}\label{left_int}
\I{\alpha}{a^+}f(x) = \dfrac{\rho^{1-\alpha}}{\Gamma(\alpha)}\int_a^x \dfrac{\tau^{\rho-1}\,f(\tau)}{(x^\rho-\tau^\rho)^{1-\alpha}}\ d\tau.
\end{align}
This definition is a fractional generalization of the $n$-fold left integral of the form
\begin{align*}
\I{n}{a^+}f(x) = \int_a^x t_1^{\rho-1}\, dt_1 \int_a^{t_1} t_2^{\rho-1}\, dt_2\, \dots\int_a^{t_{n-1}} t_n^{\rho-1}f(t_n)\ dt_n.
\end{align*}
Similarly, the right-sided generalized fractional integral $\I{\alpha}{b^-}f$ of order $\alpha$ is defined for any real number $x<b$ by:
\begin{align}\label{right_int}
\I{\alpha}{b^-}f(x) = \dfrac{\rho^{1-\alpha}}{\Gamma(\alpha)}\int_x^b \dfrac{\tau^{\rho-1}\,f(\tau)}{(\tau^\rho-x^\rho)^{1-\alpha}}\ d\tau.
\end{align}
%
The corresponding generalized fractional derivatives to these generalized integrals are given in what follows.
\begin{defi}
Let $\alpha\in \IC$ with $Re(\alpha)\geq 0$ and $\rho > 0$. Let \mbox{$n = \lfloor Re(\alpha)\rfloor+1$} and $f\in AC^n_\gamma[a,b]$. The generalized fractional derivatives relative to the generalized integrals \eqref{left_int} and \eqref{right_int} are given for any real number $0 \leq a < x < b \leq \infty$ respectively by:
\begin{align*}
\D{\alpha}{a^+}f(x) & := \left(x^{1-\rho}\,\dfrac{d}{dx}\right)^n\left(\I{n-\alpha}{a^+}f\right)(x) \\
& \,= \dfrac{\rho^{\alpha-n+1}}{\Gamma(n-\alpha)} \left(x^{1-\rho}\,\dfrac{d}{dx}\right)^n\int_a^x \dfrac{\tau^{\rho-1}\,f(\tau)}{(x^\rho-\tau^\rho)^{\alpha-n+1}}\ d\tau
\end{align*}
and
\begin{align*}
\D{\alpha}{b^-}f(x) & := \left(-x^{1-\rho}\,\dfrac{d}{dx}\right)^n\left(\I{n-\alpha}{b^-}f\right)(x) \\
& \,= \dfrac{\rho^{\alpha-n+1}}{\Gamma(n-\alpha)} \left(-x^{1-\rho}\,\dfrac{d}{dx}\right)^n\int_x^b \dfrac{\tau^{\rho-1}\,f(\tau)}{(\tau^\rho-x^\rho)^{\alpha-n+1}}\ d\tau.
\end{align*}
\end{defi}
It is worth noting that $\D{0}{a^+}$ and $\D{0}{b^-}$ simply reduce to the identity operator. In \cite{Jar17}, it is proven that these fractional operators are well defined on $AC^n_\gamma[a,b]$. Moreover, the generalized fractional integrals and derivatives satisfy the semigroup, the composition and the inverse property. More precisely, we have \cite{Kat11,Kat14}
\begin{prop}\label{prop_semigroup}
Let $\alpha>0$, $\beta>0$, $0 < a < b \leq +\infty$ and $\rho>0$. Then
\begin{enumerate}
\item $\I{\alpha}{a^+}\, \I{\beta}{a^+}f = \I{\alpha+\beta}{a^+}f$ and $\D{\alpha}{a^+}\, \D{\beta}{a^+}f = \D{\alpha+\beta}{a^+}f$.
\item $\D{\alpha}{a^+}\, \I{\beta}{a^+}f = \I{\beta-\alpha}{a^+}f\ $ if $\ \alpha < \beta$.
\item $\D{\alpha}{a^+}\, \I{\alpha}{a^+}f = f$. 
\end{enumerate}
\end{prop}
All theses properties remain true in case one replace the left-sided by right-sided generalized integrals and derivatives. The following Theorem gives a link between the generalized fractional integrals and derivatives and those of Riemann-Liouville and Hadamard.
\begin{thm}\label{thm_cv_rho}
Let $\alpha\in \IC$ such that $Re(\alpha)\geq 0$, $n= \lfloor Re(\alpha) \rceil$ and $\rho>0$. Then for $x>a$
\begin{enumerate}
\item $\dis{\lim_{\rho\to 1} \I{\alpha}{a^+}f(x) = \dfrac{1}{\Gamma(\alpha)}\int_a^x \dfrac{f(\tau)}{(x-\tau)^{1-\alpha}}\,d\tau}$.
\item $\dis{\lim_{\rho\to 0^+} \I{\alpha}{a^+}f(x) = \dfrac{1}{\Gamma(\alpha)}\int_a^x \left(\log\dfrac{x}{\tau}\right)^{\alpha-1}f(\tau)\,\dfrac{d\tau}{\tau}}$.
\item $\dis{\lim_{\rho\to 1} \D{\alpha}{a^+}f(x) = \dfrac{1}{\Gamma(n-\alpha)}\left(\dfrac{d}{dx}\right)^n\int_a^x \dfrac{f(\tau)}{(x-\tau)^{\alpha-n+1}}\,d\tau}$.
\item $\dis{\lim_{\rho\to 0^+} \D{\alpha}{a^+}f(x) = \dfrac{1}{\Gamma(\alpha)}\left(x\dfrac{d}{dx}\right)^n\int_a^x \left(\log\dfrac{x}{\tau}\right)^{n-\alpha-1}f(\tau)\,\dfrac{d\tau}{\tau}}$.
\end{enumerate}

\end{thm}
\begin{proof}
See \cite[Theorem 4.1]{Kat14}.
\end{proof}
Apart from that the Riemann--Liouville fractional derivatives of a constant is not equal to zero, there is no known physical interpretation for the expressions $\lim_{x\to 0^+}\left(\mathcal{D}^{\alpha-k}_{0^+}f\right)(x)$, $k=1,\dots,n$. For these reasons, the Riemann--Liouville fractional derivatives are not well suited to model applied problems \cite{Luch07}. To overcome these difficulties, Caputo introduced a modification to the Riemann--Liouville fractional derivative \cite{Cap67} which is commonly known as the Caputo-type (or the Caputo modified, or simply the Caputo) fractional derivative. In \cite{Kat11}, a generalized Caputo-type fractional derivative was introduced and extensively investigated in \cite{Jar17}. Below we recall some definitions and properties of these generalized Caputo-type fractional derivatives.
\begin{defi}
Let $Re(\alpha) \geq 0$ and $n = \lfloor Re(\alpha) \rceil$. Let $f \in AC^n_\gamma[a, b]$, where $0 < a < b < +\infty$. The left and right generalized Caputo fractional derivatives of order $\alpha$ of $f$ are defined by:
\begin{align*}
\C{\alpha}{a^+}f(x) := \D{\alpha}{a^+}\left(f(t) - \sum_{k=0}^{n-1} \dfrac{\left(\gamma^{k}f\right)(a)}{k!}\left(\dfrac{t^\rho-a^\rho}{\rho}\right)^k\right)(x) \\ 
\C{\alpha}{b^-}f(x) := \D{\alpha}{b^-}\left(f(t) - \sum_{k=0}^{n-1} \dfrac{(-1)^k\,\left(\gamma^{k}f\right)(b)}{k!}\left(\dfrac{b^\rho-t^\rho}{\rho}\right)^k\right)(x) 
\end{align*}
(we recall that $\gamma := x^{1-\rho}\,\dfrac{d}{dx}$, $\gamma^k = \underbrace{\gamma\circ\dots\circ\gamma}_{k \text{ times}}$ and $\gamma^0 = id$).
\end{defi}
We have the following result.
\begin{prop}\label{prop_gen_der_express}
Let $Re(\alpha) > 0$ and $n = \lfloor Re(\alpha) \rceil$. Let $f \in AC^n_\gamma[a, b]$, where $0 < a < b < +\infty$.
\begin{enumerate}
\item If $\alpha \not \in \IN_0$, then
\begin{align}\label{cap_der_left_int_reel}
\C{\alpha}{a^+}f(x) = \dfrac{\rho^{\alpha-n+1}}{\Gamma(n-\alpha)}\int_a^x \dfrac{\tau^{\rho-1} \left(\gamma^n f\right)(\tau)}{\left(x^\rho-\tau^\rho\right)^{\alpha-n+1}}\ d\tau = \I{n-\alpha}{a^+}\left(\gamma^{n}f\right)(x)
\end{align}
and
\begin{align*}
\C{\alpha}{b^-}f(x) = \dfrac{(-1)^n\,\rho^{\alpha-n+1}}{\Gamma(n-\alpha)}\int_x^b \dfrac{\tau^{\rho-1} \left(\gamma^n f\right)(\tau)}{\left(\tau^\rho-x^\rho\right)^{\alpha-n+1}}\ d\tau = (-1)^n\,\I{n-\alpha}{b^-}\left(\gamma^{n}f\right)(x).
\end{align*}
\item If $\alpha=n \in \IN_0$, then
\begin{align*}
\C{n}{a^+}f(x) = \D{n}{a^+}f(x) = \left(\gamma^n\,f\right)(x)
\end{align*}
and
\begin{align*}
\C{n}{b^-}f(x) = \D{n}{b^-}f(x) = (-1)^n\left(\gamma^n\,f\right)(x).
\end{align*}
In particular
\begin{align*}
\C{0}{a^+}f = \D{0}{a^+}f = f \quad \text{and} \quad \C{0}{b^-}f = \D{0}{b^-}f = f.
\end{align*}
\end{enumerate}
\end{prop}
The generalized Caputo-type fractional derivatives satisfy the following composition properties.
\begin{prop}\label{prop_comp_cap_I_der}
Let $\alpha\in \IC$ and $n= \lfloor Re(\alpha) \rceil$. For $f\in AC^n_\gamma[a,b]$ we have
\begin{align}\label{comp_cap_I_der_left}
\I{\alpha}{a^+}\C{\alpha}{a^+}f(x) = f(x) - \sum_{k=0}^{n-1} \dfrac{\left(\gamma^k f\right)(a)}{k!}\left(\dfrac{x^\rho-a^\rho}{\rho}\right)^k
\end{align}
and
\begin{align}\label{comp_cap_I_der_right}
\I{\alpha}{b^-}\C{\alpha}{b^-}f(x) = f(x) - \sum_{k=0}^{n-1} \dfrac{(-1)^k\,\left(\gamma^k f\right)(b)}{k!}\left(\dfrac{b^\rho-x^\rho}{\rho}\right)^k.
\end{align}
In addition, if $f\in AC^{n+m}_\delta[a,b]$ and $\alpha\geq 0$ and $\beta\geq 0$ are such that $n-1 < \alpha \leq n$ and $m-1 < \beta \leq m$, then
\begin{align*}
\C{\alpha}{a^+}\C{\beta}{a^+}f = \C{\alpha+\beta}{a^+}f  \quad \text{and}\quad  \C{\alpha}{b^-}\C{\beta}{b^-}f = \C{\alpha+\beta}{b^-}f.
\end{align*}
\end{prop}
\begin{proof}
See Theorem 3.6 and Theorem 3.7 in \cite{Jar17}.
\end{proof}

The limit cases when $\rho$ goes to $0$ or $1$ have also been investigated in \cite{Jar17}. The authors obtained the following results (see \cite[Theorem 3.11]{Jar17}).
\begin{thm}\label{thm_cv_rho_caputo}
Let $\alpha\in \IC$ such that $Re(\alpha)\geq 0$ and $n= \lfloor Re(\alpha) \rceil$. Then for $x>a$
\begin{enumerate}
\item $\dis{\lim_{\rho\to 1} \C{\alpha}{a^+}f(x) = \mathcal{C}^{\alpha}_{a^+}f(x)}$.
\item $\dis{\lim_{\rho\to 0^+} \C{\alpha}{a^+}f(x) = \Hd{\alpha}{a^+}f(x)}$.
\end{enumerate}
where $\mathcal{C}^{\alpha}_{a^+}$ is the Caputo modified Riemann-Liouville fractional derivative \cite{Kil06} and $\Hd{\alpha}{a^+}$ is the Caputo modified Hadamard fractional derivative \cite{Gam14} defined by 
\begin{align*}
\mathcal{C}^\alpha_{a^+}f(x):= \dfrac{1}{\Gamma(n-\alpha)}\,\int_a^x \dfrac{f^{(n)}(\tau)}{(x-\tau)^{\alpha-n+1}}\,d\tau
\end{align*}
and
\begin{align}\label{Caputo modified Hadamard}
\Hd{\alpha}{a^+}f(x):= \dfrac{1}{\Gamma(n-\alpha)}\,\int_a^x \left(\log\dfrac{x}{\tau}\right)^{n-\alpha-1} \left[\left(\tau\dfrac{d}{d\tau}\right)^nf\right](\tau)\,\dfrac{d\tau}{\tau}.
\end{align}
\end{thm}


\section{Generalized Taylor formulas}\label{sec_taylor}
We provide in this section generalized Taylor formulas involving the generalized fractional derivatives in Caputo sens. We first establish an analogous theorem to the mean value theorem.
\begin{thm}\label{thm_mean_value}
Let $0< \alpha \leq 1$ and $f\in C[a,b]$ such that $\C{\alpha}{a^+}f\in C(a,b)$ (resp. $\C{\alpha}{b^-}f\in C(a,b)$). Then for all $x\in [a,b]$ there exists $\xi\in\, ]a,x[$ (resp. $\xi\in\, ]x,b[$) such that
\begin{align}\label{taf_left}
f(x) = f(a) + \dfrac{1}{\Gamma(\alpha+1)}\left(\dfrac{x^\rho-a^\rho}{\rho}\right)^\alpha\C{\alpha}{a^+}f(\xi).
\end{align}
(resp.
\begin{align}\label{taf_right}
f(x) = f(b) + \dfrac{1}{\Gamma(\alpha+1)}\left(\dfrac{b^\rho-x^\rho}{\rho}\right)^\alpha\C{\alpha}{b^-}f(\xi)).
\end{align}
\end{thm}
\begin{proof}
Since $0< \alpha \leq 1$ we have by Proposition \ref{prop_comp_cap_I_der} that
\begin{align}\label{I=f(x)-f(a)}
\I{\alpha}{a^+}\C{\alpha}{a^+}f(x) = f(x) - f(a).
\end{align}
Now, using equation \eqref{left_int} and the integral mean value theorem, we obtain
\begin{align}\label{I=Dxi}
\I{\alpha}{a^+}\C{\alpha}{a^+}f(x) & = \dfrac{\rho^{1-\alpha}}{\Gamma(\alpha)}\int_a^x \dfrac{\tau^{\rho-1}\,\C{\alpha}{a^+}f(\tau)}{(x^\rho-\tau^\rho)^{1-\alpha}}\ d\tau \nonumber\\
& = \dfrac{\rho^{1-\alpha}}{\Gamma(\alpha)}\,\C{\alpha}{a^+}f(\xi)\,\int_a^x \dfrac{\tau^{\rho-1}}{(x^\rho-\tau^\rho)^{1-\alpha}}\ d\tau \nonumber\\
& = \dfrac{1}{\Gamma(\alpha+1)}\left(\dfrac{x^\rho-a^\rho}{\rho}\right)^\alpha\C{\alpha}{a^+}f(\xi). 
\end{align}
with $\xi \in\, ]a,x[$. Finally, the result follows by equating \eqref{I=f(x)-f(a)} and \eqref{I=Dxi}. Equation \eqref{taf_right} can be obtained by a same reasoning.
\end{proof}
%
%
Let us notice that in case $\alpha=\rho=1$, the classical mean value theorem is recovered. Now, we extend this result to arbitrary order $m\in \IN$. We need the following result.
\begin{prop}\label{prop_diff_I}
Let $0< \alpha \leq 1$. Let $j\in \IN$ and suppose $\C{j\alpha}{a^+}f$ and $\C{(j+1)\alpha}{a^+}f$ belong to $C[a,b]$ (resp. $\C{j\alpha}{b^-}f$ and $\C{(j+1)\alpha}{b^-}f$ belong to $C[a,b]$), then
\begin{align*}
\I{j\alpha}{a^+}\C{j\alpha}{a^+}f(x) - \I{(j+1)\alpha}{a^+}\C{(j+1)\alpha}{a^+}f(x) = \dfrac{1}{\Gamma(j\alpha+1)}\left(\dfrac{x^\rho-a^\rho}{\rho}\right)^{j\alpha}\C{j\alpha}{a^+}f(a)
\end{align*}
(resp.
\begin{align*}
\I{j\alpha}{b^-}\C{j\alpha}{b^-}f(x) - \I{(j+1)\alpha}{b^-}\C{(j+1)\alpha}{b^-}f(x) = \dfrac{1}{\Gamma(j\alpha+1)}\left(\dfrac{b^\rho-x^\rho}{\rho}\right)^{j\alpha}\C{j\alpha}{b^-}f(b)),
\end{align*}
where $\C{j\alpha}{a^+}f := \C{\alpha}{a^+}\circ\C{\alpha}{a^+}\circ\dots\circ\C{\alpha}{a^+}f\ $ ($j$ times).
\end{prop}
\begin{proof}
We only prove the first identity, since the second one can be obtained in the same manner. Using the semigroup properties of the generalized fractional integrals and derivatives given in Proposition \ref{prop_semigroup} and Proposition \ref{prop_comp_cap_I_der}, we get
\begin{align*}
\I{j\alpha}{a^+}\C{j\alpha}{a^+}f(x) - \I{(j+1)\alpha}{a^+}\C{(j+1)\alpha}{a^+}f(x) & = \I{j\alpha}{a^+}\left(\C{j\alpha}{a^+}f - \I{\alpha}{a^+}\C{(j+1)\alpha}{a^+}f\right)(x) \\
& = \I{j\alpha}{a^+}\left(\C{j\alpha}{a^+}f - \left(\I{\alpha}{a^+}\C{\alpha}{a^+}\right)\C{j\alpha}{a^+}f\right)(x) \\
& = \I{j\alpha}{a^+}\left(\C{j\alpha}{a^+}f(a)\right)(x)
\end{align*}
where the last equality holds true in virtue of equation \eqref{comp_cap_I_der_left}. Finally, using the definition of the generalized fractional integral, we obtain
\begin{align*}
\I{j\alpha}{a^+}\C{j\alpha}{a^+}f(x) - \I{(j+1)\alpha}{a^+}\C{(j+1)\alpha}{a^+}f(x) & = \left(\I{j\alpha}{a^+}\left(1\right)(x)\right)\,\C{j\alpha}{a^+}f(a) \\
& = \dfrac{1}{\Gamma(j\alpha+1)}\left(\dfrac{x^\rho-a^\rho}{\rho}\right)^{j\alpha}\C{j\alpha}{a^+}f(a).
\end{align*}
\end{proof}
We are now able to state the generalized Taylor formula by means of generalized Caputo-type fractional derivatives.
\begin{thm}\label{thm_DL}
Let $0\leq a < b < +\infty$. Let $0< \alpha \leq 1$ and let $m$ be an arbitrary non-negative integer. Suppose $\C{j\alpha}{a^+}f\in C[a,b]$ (resp. $\C{j\alpha}{b^-}f\in C[a,b]$) for $j=0,\, 1, \dots,m+1$, then the generalized Taylor-Lagrange formula involving the generalized Caputo-type fractional derivatives writes
\begin{align*}
f(x) = \sum_{j=0}^m \left(\dfrac{x^\rho-a^\rho}{\rho}\right)^{j\alpha} \dfrac{\C{j\alpha}{a^+}f(a)}{\Gamma(j\alpha+1)}+ \left(\dfrac{x^\rho-a^\rho}{\rho}\right)^{(m+1)\alpha}\dfrac{\C{(m+1)\alpha}{a^+}f(\xi)}{\Gamma((m+1)\alpha+1)}
\end{align*}
(resp.
\begin{align*}
f(x) = \sum_{j=0}^m \left(\dfrac{b^\rho-x^\rho}{\rho}\right)^{j\alpha} \dfrac{\C{j\alpha}{b^-}f(b)}{\Gamma(j\alpha+1)}+ \left(\dfrac{b^\rho-x^\rho}{\rho}\right)^{(m+1)\alpha}\dfrac{\C{(m+1)\alpha}{b^-}f(\xi)}{\Gamma((m+1)\alpha+1)})
\end{align*}
where $\xi \in\, ]a,x[$ (resp. $\xi \in\, ]x,b[$) and $\C{j\alpha}{a^+}f := \C{\alpha}{a^+}\circ\dots\circ\C{\alpha}{a^+}f\ $ ($j$ times).
\end{thm}
\begin{proof}
Using Proposition \ref{prop_diff_I} we have
\begin{align*}
\sum_{j=0}^m \I{j\alpha}{a^+}\C{j\alpha}{a^+}f(x) - \I{(j+1)\alpha}{a^+}\C{(j+1)\alpha}{a^+}f(x) & = \sum_{j=0}^m \dfrac{1}{\Gamma(j\alpha+1)}\left(\dfrac{x^\rho-a^\rho}{\rho}\right)^{j\alpha}\C{j\alpha}{a^+}f(a).
\end{align*}
It follows
\begin{align*}
\I{0}{a^+}\C{0}{a^+}f(x) - \I{(m+1)\alpha}{a^+}\C{(m+1)\alpha}{a^+}f(x) & = \sum_{j=0}^m \dfrac{1}{\Gamma(j\alpha+1)}\left(\dfrac{x^\rho-a^\rho}{\rho}\right)^{j\alpha}\C{j\alpha}{a^+}f(a),
\end{align*}
that is
\begin{align}\label{f(x)=sum+reste}
f(x) = \sum_{j=0}^m \dfrac{1}{\Gamma(j\alpha+1)}\left(\dfrac{x^\rho-a^\rho}{\rho}\right)^{j\alpha}\C{j\alpha}{a^+}f(a) + \I{(m+1)\alpha}{a^+}\C{(m+1)\alpha}{a^+}f(x).
\end{align}
Now we make use of the same arguments as in the proof of Theorem \ref{thm_mean_value}. In fact, replacing $\alpha$ by $(m+1)\alpha$ in equation \eqref{I=Dxi}, we get
\begin{align*}
\I{(m+1)\alpha}{a^+}\C{(m+1)\alpha}{a^+}f(x) = \dfrac{1}{\Gamma((m+1)\alpha+1)}\left(\dfrac{x^\rho-a^\rho}{\rho}\right)^{(m+1)\alpha}\C{(m+1)\alpha}{a^+}f(\xi)
\end{align*}
with $\xi\in\, ]a,x[$. The proof of the second identity involving the right-sided generalized fractional derivatives is similar. This ends the proof.
\end{proof}
\begin{cor}\label{cor_DL_reste_int}
Under the same assumptions of Theorem \ref{thm_DL}, the generalized Taylor expansion of $f$ could also be written with remainder in integral form. We have
\begin{align*}
f(x) & = \sum_{j=0}^m \left(\dfrac{x^\rho-a^\rho}{\rho}\right)^{j\alpha} \dfrac{\C{j\alpha}{a^+}f(a)}{\Gamma(j\alpha+1)} + \dfrac{\rho^{1-(m+1)\alpha}}{\Gamma((m+1)\alpha)}\int_a^x \dfrac{\tau^{\rho-1}\,\C{(m+1)\alpha}{a^+}f(\tau)}{(x^\rho-\tau^\rho)^{1-(m+1)\alpha}}\ d\tau,
\end{align*}
and similarly
\begin{align*}
f(x) & = \sum_{j=0}^m \left(\dfrac{b^\rho-x^\rho}{\rho}\right)^{j\alpha} \dfrac{\C{j\alpha}{b^-}f(b)}{\Gamma(j\alpha+1)} + \dfrac{\rho^{1-(m+1)\alpha}}{\Gamma((m+1)\alpha)}\int_x^b \dfrac{\tau^{\rho-1}\,\C{(m+1)\alpha}{b^-}f(\tau)}{(\tau^\rho-x^\rho)^{1-(m+1)\alpha}}\ d\tau.
\end{align*}
\end{cor}
\begin{proof}
Direct consequence of equations \eqref{f(x)=sum+reste} and \eqref{left_int}.
\end{proof}
\begin{rmrk}
If $0< \alpha \leq 1$ then Proposition \ref{prop_comp_cap_I_der} becomes a particular case of Theorem \ref{thm_DL} (or Corollary \ref{cor_DL_reste_int}) corresponding to the zeroth order Taylor expansion of $f$ (i.e $m=0$). Indeed, the result stated in Proposition \ref{prop_comp_cap_I_der} does not provide further information on the Taylor expansion of $f$ since the summation in the right hand side of equation \eqref{comp_cap_I_der_left} depends on $\alpha$, and for $0< \alpha \leq 1$ it simply reduces to $f(a)$. Notice that the classical Taylor expansion formula can also be recovered when taking $\alpha=\rho=1$ in Theorem \ref{thm_DL} or Corollary \ref{cor_DL_reste_int}.
\end{rmrk}
\begin{rmrk}\label{rem_Cjalpha}
The direct evaluation of the fractional derivatives $\C{j\alpha}{a^+}f(a)$, $j=0,\dots,m$ that appear in Theorem \ref{thm_DL} (or Corollary \ref{cor_DL_reste_int}) can be a difficult task. Alternatively, one can easily evaluate them using the identities given in Proposition \ref{prop_gen_der_express}. Indeed, applying the integral mean value theorem to equation \eqref{cap_der_left_int_reel}, one can obtain after simplifications
\begin{align*}
\C{j\alpha}{a^+}f(a) = 
\left\{
\begin{array}{ll}
\dfrac{1}{\Gamma(n-j\alpha+1)}\dis{\lim_{x\to a^+}} \left(\dfrac{x^\rho-a^\rho}{\rho}\right)^{n-j\alpha}\left(\gamma^n f\right)(x) & \text{ if }\ j\alpha\not\in \IN_0,\\ & \\
\left(\gamma^n f\right)(a) & \text{ if }\ j\alpha\in \IN_0
\end{array}
\right.
\end{align*}
with $n=\lfloor j\alpha \rceil$. Analogous formula can be obtained for $\C{j\alpha}{b^-}f(b)$. In \ref{Appendix_gamma} we give an algorithm to recursively evaluate the term $\left(\gamma^n f\right)(x)$.
\end{rmrk}
In the limit case when $\rho$ goes to zero, we obtain the following result.
\begin{cor}
Let $0< a < b < +\infty$. Let $0< \alpha \leq 1$ and let $m$ be an arbitrary non-negative integer. Suppose $\Hd{j\alpha}{a^+}f\in C[a,b]$ for $j=0,\, 1, \dots,m+1$, then for $x>a$
\begin{align*}
f(x) = \sum_{j=0}^m \left(\log\dfrac{x}{a}\right)^{j\alpha} \dfrac{\Hd{j\alpha}{a^+}f(a)}{\Gamma(j\alpha+1)}+ \left(\log\dfrac{x}{a}\right)^{(m+1)\alpha}\dfrac{\Hd{(m+1)\alpha}{a^+}f(\xi)}{\Gamma((m+1)\alpha+1)}
\end{align*}
where $\xi\in\, ]a,x[$ and $\Hd{\alpha}{a^+}$ is Caputo modified Hadamard fractional derivative defined by \eqref{Caputo modified Hadamard}.
\end{cor}
\begin{proof}
The proof follows from the L'Hospital rule and Theorem \ref{thm_cv_rho_caputo}. 
\end{proof}
%

\section{Approximation of functions with M\"untz Polynomials}\label{sec_muntz}
We utilize of the previous results in order to derive some approximation formulas of functions at a given point.
\begin{prop}\label{prop_approx_function}
Let $f$ be a given function such that $\C{j\alpha}{a^+}f\in C[a,b]$ for all $j=0,1\dots m+1$, where $0< \alpha \leq 1$ and $m\in \IN_0$. Then for all $x\in [a,b]$ we have
\begin{align}\label{Muntz-expansion}
f(x) \simeq \Lambda_m(x) = \sum_{j=0}^m \left(\dfrac{x^\rho-a^\rho}{\rho}\right)^{j\alpha} \dfrac{\C{j\alpha}{a^+}f(a)}{\Gamma(j\alpha+1)}.
\end{align}
Moreover, the interpolation error can be expressed as
\begin{align}\label{Muntz-error}
e_m(x) : = f(x) - \Lambda_m(x) = \left(\dfrac{x^\rho-a^\rho}{\rho}\right)^{(m+1)\alpha}\dfrac{\C{(m+1)\alpha}{a^+}f(\xi)}{\Gamma((m+1)\alpha+1)}
\end{align}
with $\xi \in\, ]a,x[$.
\end{prop}
The proof of proposition \ref{prop_approx_function} follows directly from Theorem \ref{thm_DL}. 
\begin{rmrk}
When $a=0$, Proposition \ref{prop_approx_function} gives the Taylor expansion of $f$ in terms of M\"u{}ntz Polynomials \cite{Alm07,Bor94}, which are a generalization of the standard polynomials to real (non necessarily integer) exponent sequences. In our case, these exponents take the form $\lambda_j = \beta\,j$ where $\beta=\rho\,\alpha >0$ and  $j=0,\, 1, \dots m$. Let us notice that the particular case when $\rho = 1$ (i.e. $0 < \beta \leq 1$) has been studied in \cite{Odi07}.
\end{rmrk}
Now, we apply the previous approximation to some classical functions.\\

\noindent \textbf{Example 1:} Consider $f(x) = \exp(x^\beta)$ with $\beta>0$.\\

The simplest way to find the generalized Taylor expansion of $f$ in the neighborhood of $a=0$ is to choose $\rho = \beta$ and $\alpha =1$ in equation \eqref{Muntz-expansion}. Using Remark \ref{rem_Cjalpha} we find
\begin{align*}
\C{j}{0^+}f(x) = \left(\gamma^jf\right)(x) = \rho^jf(x) = \beta^j\,\exp(x^\beta),
\end{align*}
and hence $\C{j}{0^+}f(0) = \beta^j$. It follows
\begin{align*}
f(x) & = \sum_{j=0}^m \left(\dfrac{x^\beta}{\beta}\right)^j\dfrac{\beta^j}{\Gamma(j+1)} + e_m(x) \\
     & = \sum_{j=0}^m \dfrac{x^{\beta j}}{j!} + \dfrac{x^{\beta(m+1)}}{(m+1)!}\exp(\xi^\beta)
\end{align*}
with $\xi \in\, ]0,x[$.\\

\noindent \textbf{Example 2:} The generalized Mittag-Leffler function \cite{Kil06} is defined for complex $z\in\IC$, $\lambda$, $\mu$, $\nu \in \IC$ with $Re(\lambda)>0$ by
\begin{align*}
E_{\lambda,\mu}^\nu(z) = \sum_{k=0}^\infty \dfrac{(\nu)_k}{\Gamma(\lambda k+\mu)}\dfrac{z^k}{k!}
\end{align*}
where $(\nu)_k$ is the Pochhammer symbol. This function (and particularly the classical Mittag-Leffler function $E_\lambda:=E^1_{\lambda,1}$) is a powerful tool to solve fractional differential equations \cite{Pod98,Sch96}. Using equations \eqref{Muntz-expansion} and \eqref{Muntz-error} one can deduce that for $\beta>0$
\begin{align*}
E_{\lambda,\mu}^\nu(z^\beta) = \left(\sum_{j=0}^m\left(\dfrac{z^{\beta}}{\beta}\right)^j\dfrac{1}{\Gamma(j+1)} \sum_{k=0}^\infty \dfrac{(\nu)_k}{\Gamma(\lambda k+\mu)}\dfrac{\mathcal{C}^{j,\beta}_{0^+}f_k(0)}{k!}\right) + e_m(z)
\end{align*}
where $f_k(z) = z^{\beta k}$. A direct computation (or using \ref{Appendix_gamma}) shows that 
\begin{align*}
\mathcal{C}^{j,\beta}_{0^+}f_k(z) = (\gamma^jf_k)(z) = \left\{
\begin{array}{ll}
\beta^j \,\dfrac{\Gamma(k+1)}{\Gamma(k-j+1)} z^{\beta(k-j)} & \text{ if }\ j\leq k \\
0 & \text{ if }\ j > k .
\end{array}
\right.
\end{align*}
Hence
\begin{align*}
\mathcal{C}^{j,\beta}_{0^+}f_k(0) = \beta^k\,\Gamma(k+1)\delta_{jk}
\end{align*}
 with $\delta_{jk}$ is the Kronecker delta. Using the identity $(\nu)_{p+q} = (\nu)_p\,(\nu+p)_q$, for $p,q\in \IN_0$, one can deduce after simplifications that
\begin{align*}
E_{\lambda,\mu}^\nu(z^\beta) = \sum_{j=0}^m \dfrac{(\nu)_j}{\Gamma(\lambda j+\mu)} \dfrac{z^{\beta j}}{j!} + \dfrac{(\nu)_{m+1}}{(m+1)!}\,z^{\beta(m+1)}\,E_{\lambda,\lambda(m+1)+\mu}^{\nu+m+1}(\xi^\beta)
\end{align*}
with $\xi\in\, ]0,x[$.\\

\noindent \textbf{Example 3:} The generalized Hypergeometric series \cite{Kil06} are defined by
\begin{align*}
{}_pF_q(a_1,\dots,a_p;b_1,\dots,b_q;z) = \sum_{k=0}^\infty \dfrac{(a_1)_k\dots (a_p)_k}{(b_1)_k\dots (b_q)_k}\dfrac{z^k}{k!}
\end{align*}
where $a_i,b_j\in\IC$, $b_j\neq 0, -1, -2, \dots$, ($i=1\dots p$, $j=1\dots q$). This series
is absolutely convergent for all values of $z\in \IC$ if $p < q$. When $p = q + 1$, the series is absolutely convergent for $|z| < 1$. Using the same arguments of example 2, one can obtain for $\beta > 0$
\begin{align*}
& {}_pF_q(a_1,\dots,a_p;b_1,\dots,b_q;z^\beta) = \sum_{j=0}^m \dfrac{(a_1)_j\dots (a_p)_j}{(b_1)_j\dots (b_q)_j}\dfrac{z^{\beta j}}{j!} \ + \\
& \ \dfrac{(a_1)_{m+1}\dots (a_p)_{m+1}}{(b_1)_{m+1}\dots (b_q)_{m+1}}\dfrac{z^{\beta(m+1)}}{(m+1)!}\ {}_pF_q(a_1+m+1,\dots,a_p+m+1;b_1+m+1,\dots,b_q+m+1;\xi^{\beta})
\end{align*}
with $\xi\in\, ]0,x[$.

\section{Series solutions to fractional differential equations}\label{sec_series}
In this section, we investigate the solutions of some fractional differential equations (fde) in terms of series. The following result will be helpful.
\begin{prop}\label{prop_der-poly}
Let $Re(\alpha) > 0$ and $n = \lfloor Re(\alpha) \rfloor + 1$. Then
\begin{align*}
&\C{\alpha}{a^+}\left(\dfrac{x^\rho-a^\rho}{\rho}\right)^\beta = \dfrac{\Gamma(1+\beta)}{\Gamma(1+\beta-\alpha)}\left(\dfrac{x^\rho-a^\rho}{\rho}\right)^{\beta-\alpha},& Re(\beta) > n-1.\\
&\C{\alpha}{a^+}\left(\dfrac{x^\rho-a^\rho}{\rho}\right)^k = 0, & k=0,\, 1,\dots,\, n-1.
\end{align*}
\end{prop}
\begin{proof}
See Lemma 3.8 and 3.9 in \cite{Jar17}
\end{proof}

\noindent \textbf{Example 1:} Consider the following fde
\begin{align}\label{fde_example1}
\left\{
\begin{array}{l}
\C{\alpha}{a^+}u(x) = \lambda\,u(x), \ x > a\\
u(a) = u_a,
\end{array}
\right.
\end{align}
where $0 < \alpha \leq 1$, $a \geq 0$, $\rho>0$, $\lambda\in \IR-\{0\}$, $u_a\in \IR$ and $u$ is the unknown function. Suppose that $u$ can be written in series form
\begin{align}\label{express_u}
u(x) = \sum_{n=0}^{+\infty} c_n\,(x^\rho-a^\rho)^{n\alpha}
\end{align}
Using Proposition \ref{prop_der-poly}, we obtain
\begin{align}
\C{\alpha}{a^+}u(x) & = \rho^\alpha\sum_{n=1}^{+\infty} c_n\,\dfrac{\Gamma(1+n\alpha)}{\Gamma(1+(n-1)\alpha)}(x^\rho-a^\rho)^{(n-1)\alpha} \nonumber \\
& = \rho^\alpha\sum_{n=0}^{+\infty} c_{n+1}\,\dfrac{\Gamma(1+(n+1)\alpha)}{\Gamma(1+n\alpha)}(x^\rho-a^\rho)^{n\alpha}. \label{Calpha_u}
\end{align}
Plugging equation \eqref{Calpha_u} into \eqref{fde_example1} yields
\begin{align*}
c_{n+1} = \rho^{-\alpha} \dfrac{\Gamma(1+n\alpha)}{\Gamma(1+(n+1)\alpha)}\lambda\,c_n, \quad \forall\ n \in \IN_0
\end{align*}
which immediately gives
\begin{align}\label{cn}
c_n = \left(\rho^{-\alpha} \lambda\right)^n \dfrac{\Gamma(1)}{\Gamma(1+n\alpha)} c_0.
\end{align}
Using the initial condition $c_0 = u_a$, we finally get by \eqref{express_u} and \eqref{cn}
\begin{align*}
u(x) = u_a\,E_\alpha{\Big(}\lambda \rho^{-\alpha} (x^\rho-a^\rho)^\alpha{\Big)}
\end{align*}
where $E_\alpha$ is the Mittag--Leffler function.\\

\noindent \textbf{Example 2:} Consider the following fde
\begin{align}\label{fde_example2}
\left\{
\begin{array}{l}
\C{2\alpha}{0^+}u(x) -2\,\C{\alpha}{0^+}u(x) + u(x) = 0, \ x > 0\\
u(0) = u_0, \ \ \C{\alpha}{0^+}u(0) = u_1
\end{array}
\right.
\end{align}
where $0 < \alpha \leq 1$, $\rho>0$, $u_0$ and $u_1$ in $\IR$ and $u$ is the unknown function. Suppose that $u$ can be written in series form
\begin{align}\label{express_u_2}
u(x) = \sum_{n=0}^{+\infty} c_n\,(x^\rho)^{n\alpha}
\end{align}
Using Proposition \ref{prop_der-poly}, we obtain
\begin{align}
\C{\alpha}{0^+}u(x) & = \rho^\alpha\sum_{n=0}^{+\infty} c_{n+1}\,\dfrac{\Gamma(1+(n+1)\alpha)}{\Gamma(1+n\alpha)}(x^\rho)^{n\alpha} \label{Calpha_u_2}
\end{align}
and
\begin{align}
\C{2\alpha}{0^+}u(x) := \C{\alpha}{0^+}\left(\C{\alpha}{0^+}u\right)(x)  = \rho^{2\alpha}\sum_{n=0}^{+\infty} c_{n+2}\,\dfrac{\Gamma(1+(n+2)\alpha)}{\Gamma(1+n\alpha)}(x^\rho)^{n\alpha}. \label{C2alpha_u_2}
\end{align}
Plugging equations \eqref{Calpha_u_2} and \eqref{C2alpha_u_2} into \eqref{fde_example2} yields
\begin{align*}
\rho^{2\alpha} c_{n+2}\dfrac{\Gamma(1+(n+2)\alpha)}{\Gamma(1+n\alpha)} - 2\,\rho^{\alpha} \dfrac{\Gamma(1+(n+1)\alpha)}{\Gamma(1+n\alpha)}+c_n = 0, \quad \forall\ n \in \IN_0.
\end{align*}
Denote $d_n = c_n\,\Gamma(1+n\alpha)$. It follows
\begin{align*}
\rho^{2\alpha} d_{n+2} - \rho^{\alpha} d_{n+1} + d_n = 0, \quad \forall\ n \in \IN_0,
\end{align*}
which gives
\begin{align*}
d_n = \Big{(}d_0+\left(\rho^{\alpha}d_1-d_0\right)n\Big{)}\rho^{-n\alpha}
\end{align*}
or equivalently
\begin{align}\label{cn_2}
c_n = \dfrac{c_0+\left(\rho^{\alpha}\Gamma(1+\alpha)c_1-c_0\right)n}{\Gamma(1+n\alpha)}\rho^{-n\alpha}.
\end{align}
Using the initial conditions $u_0 = c_0$ and $u_1 = \rho^{\alpha}\Gamma(1+\alpha)c_1$, we finally get by \eqref{express_u_2} and \eqref{cn_2}
\begin{align*}
u(x) & = \sum_{n=0}^{+\infty} \dfrac{u_0+\left(u_1-u_0\right)n}{\Gamma(1+n\alpha)}\left(\dfrac{x^\rho}{\rho}\right)^{n\alpha}\\
& = u_0\sum_{n=0}^{+\infty} \dfrac{\left(\dfrac{x^\rho}{\rho}\right)^{n\alpha}}{\Gamma(1+n\alpha)} 
+ \dfrac{u_1-u_0}{\alpha}\sum_{n=1}^{+\infty} \dfrac{\left(\dfrac{x^\rho}{\rho}\right)^{n\alpha}}{\Gamma(n\alpha)}\\
& = u_0\sum_{n=0}^{+\infty} \dfrac{\left(\dfrac{x^\rho}{\rho}\right)^{n\alpha}}{\Gamma(1+n\alpha)} 
+ \dfrac{u_1-u_0}{\alpha}\left(\dfrac{x^\rho}{\rho}\right)\sum_{n=0}^{+\infty} \dfrac{\left(\dfrac{x^\rho}{\rho}\right)^{n\alpha}}{\Gamma(n\alpha+\alpha)}\\
& = u_0\,E_\alpha\left(\rho^{-\alpha}x^{\rho\alpha}\right) + \dfrac{u_1-u_0}{\alpha}\left(\dfrac{x^\rho}{\rho}\right)E_{\alpha,\alpha}\left(\rho^{-\alpha}x^{\rho\alpha}\right).
\end{align*}

\appendix

\section{Appendix}\label{Appendix_gamma}
The evaluation of the term $\C{j\alpha}{a^+}f(a)$, $j\in \IN$, involves the computation of $\left(\gamma^jf\right)(a)$ with $\gamma^j = \underbrace{\gamma\circ \dots \circ \gamma}_{j\, \text{times}}$ and $\gamma = x^{1-\rho}\,\dfrac{d}{dx}$ (see Remark \ref{rem_Cjalpha}). We provide in the following lemma an algorithm to explicitly evaluate $\gamma^jf$ by means of the the classical derivatives of $f$ up to the $j$\up{th} order in case these latter exist.
\begin{lem}\label{lem_DL}
Let $j\in \IN$ and suppose $f$ is $j$ times continuously derivable in $a\in \IR$. Then
\begin{align}\label{gamma_jf}
\left(\gamma^jf\right)(a) = \lim_{x\to a^+} \sum_{i=1}^j \lambda_{i,j}\,x^{i-j\rho}f^{(i)}(x)
\end{align}
where $(\lambda_{i,j})_{\substack{0\leq i \leq m \\ 1\leq j \leq m}}$ is the sequence given recursively by:
\begin{align*}
\lambda_{i,j} = \left\{
\begin{array}{ll}
0 & \ \text{ if }\ \ i=0 \ \text{ or }\ i > j \\
1 & \ \text{ if }\ \ i = j \\
\lambda_{i-1,j-1} + \Big{(}i-(j-1)\rho\Big{)}\,\lambda_{i,j-1} & \ \text{ if }\ \ 1\leq i < j.
\end{array}
\right.
\end{align*}
\end{lem}
\begin{proof}
We prove \eqref{gamma_jf} by induction. The result is trivial for $j=1$. Moreover, we have
\begin{align*}
\left(\gamma^{j+1}f\right)(x) & = \gamma\left(\gamma^jf\right)(x) \\
& = x^{1-\rho}\,\dfrac{d}{dx}\left(\sum_{i=1}^j \lambda_{i,j}\,x^{i-j\rho}f^{(i)}(x)\right)\\
& = x^{1-\rho}\left(\sum_{i=1}^j (i-j\rho)\,\lambda_{i,j}\,x^{i-j\rho-1}f^{(i)}(x) + \sum_{i=1}^j \lambda_{i,j}\,x^{i-j\rho}f^{(i+1)}(x)\right)\\
& = \sum_{i=1}^j (i-j\rho)\,\lambda_{i,j}\,x^{i-(j+1)\rho}f^{(i)}(x) + \sum_{i=2}^{j+1} \lambda_{i-1,j}\,x^{i-(j+1)\rho}f^{(i)}(x)\\
& = \sum_{i=1}^{j+1} \left(\lambda_{i-1,j} + (i-j\rho)\,\lambda_{i,j}\right)x^{i-(j+1)\rho}f^{(i)}(x) \ - \underbrace{\lambda_{0,j}}_{=\ 0}\,x^{1-(j+1)\rho}f'(x)\\
& \quad -(j+1-j\rho)\,\underbrace{\lambda_{j+1,j}}_{=\ 0}\,x^{(j+1)(1-\rho)}f^{(j+1)}(x) \\
& = \sum_{i=1}^{j+1} \lambda_{i,j+1}\,x^{i-(j+1)\rho}f^{(i)}(x),
\end{align*}
and the proof is completed.
\end{proof}





\end{document}